\documentclass[11pt,reqno]{amsart}
\usepackage{amssymb,mathrsfs,color}
\usepackage{pinlabel}

\usepackage{amsmath, amsfonts, amsthm, verbatim}
\usepackage{epstopdf, mathtools}
\mathtoolsset{showonlyrefs}
\usepackage{color}
\usepackage{esint}
\usepackage{stmaryrd}

\usepackage{empheq}
\usepackage[shortlabels]{enumitem}

\newcommand{\rar}{\rightarrow}

\newcommand{\cl}{\mathcal}

\newcommand{\llb}{\llbracket}
\newcommand{\rrb}{\rrbracket}

\newcommand{\bs}[1]{\boldsymbol{#1}}

\definecolor{deepgreen}{cmyk}{1,0,1,0.5}

\newcommand{\ta}{\tau}

\newcommand{\p}{\partial}

\makeatletter

\newcommand{\Rmnum}[1]{\expandafter\@slowromancap\romannumeral #1@}
\makeatother

\newcommand{\Del}[1]{}

\numberwithin{equation}{section}

\newtheorem{thm}{Theorem}[section]

\newtheorem{prop}[thm]{Proposition}

\theoremstyle{remark}

\newtheorem{defn}[thm]{Definition}

\usepackage{tensor}

\usepackage{graphicx}
\usepackage{xcolor} 
\usepackage{tensor}
\usepackage{slashed}
\usepackage{cite}
\definecolor{green}{rgb}{0,0.8,0} 

\newcommand{\eps}{\epsilon}

\newcommand{\bbE}{\mathbb E}

\newcommand{\bbR}{\mathbb R}

\begin{document}

\title[Stretch-limited Motion]{Longitudinal Shock Waves in a Class of Semi-infinite Stretch-Limited Elastic Strings}

\author{Casey Rodriguez}

\begin{abstract}
In this paper, we initiate the study of wave propagation in a recently proposed mathematical model for stretch-limited elastic strings. We consider the longitudinal motion of a simple class of uniform, semi-infinite, stretch-limited strings under no external force with finite end held fixed and prescribed tension at the infinite end. We study a class of motions such that the string has one inextensible segment, where the local stretch is maximized, and one extensible segment. The equations of motion reduce to a simple and novel shock front problem in one spatial variable for which we prove existence and uniqueness of local-in-time solutions for appropriate initial data. We then prove the orbital asymptotic stability of an explicit two-parameter family of piece-wise constant stretched motions. If the prescribed tension at the infinite end is increasing in time, our results provide an open set of initial data launching motions resulting in the string becoming fully inextensible and tension blowing up in finite time. 
\end{abstract}

\maketitle

\section{Introduction}

The exact equations for planar motion of mechanical strings were derived by Euler (1751) and for general motion in three-dimensional space by Lagrange (1762), and they have been a rich source of research for natural philosophers and mathematicians (see \cite{AntmanBook}). In standard treatments, a string is \emph{elastic} if the tension is given as an explicit function of the stretch of the string (see \cite{Antman79, AntmanBook}). This then implies the existence of a stored energy functional which by the first law of thermodynamics implies mechanical work is not converted to heat during the string's motion (i.e. there is no dissipation). 

In a series of papers \cite{Raj_Implicit03, RajConspectus, RajElastElast}, Rajagopal recognized within the context of full three-dimensional continuum mechanics that the absence of dissipation is not predicated upon assuming the stress variable is an explicit function of the strain variable. More precisely, if we define \emph{elasticity} as the incapability of dissipation, then elastic bodies can be modeled by a wider class of relations between the stress variable and strain variable than those where the stress is a function of strain. This generalized notion of elasticity has had numerous applications in the modeling of electro and magneto-elastic bodies \cite{BustRaj_Elasto13, BustRaj_Magneto15}, fracture in brittle materials \cite{GouMallRajWalton_PlaneCrack15, BulMalekRajWal_Existence15}, gum metal \cite{Raj_GumMetal14} and many other materials (see the recent review \cite{RajBust20} for further references).
 
In \cite{Rod20}, we initiated the study of these generalized relations for strings by considering a class of \emph{stretch-limited} elastic strings. For these strings, the stretch is an explicit function of the tension of the string, but not vice versa, and the stretch is constrained between two limiting values. Therefore, segments of stretch-limited strings can neither be stretched nor compressed beyond certain lengths. Put another way, if one fixes one end and pulls hard enough on the other end of a stretch-limited string, then the string becomes completely taut and incapable of further stretching (a common experience with real strings). These properties are incapable of being modeled within the standard framework for elastic strings. In \cite{Rod20} we studied stationary configurations of finite stretch-limited strings suspended between two points under the force of gravity (catenaries) containing both inextensible segments (where the stretch is maximized) and extensible segments. 

In this work we study the dynamic longitudinal motion of a simple class of uniform, semi-infinite, tensile, stretch-limited elastic strings with finite end held fixed, prescribed tension at the infinite end (the end at infinity being pulled taut),  and containing one inextensible segment and one extensible segment. A brief overview of the formulation of the problem is as follows (see Section 2 for more details). Let $I = [0,\infty)$ parameterize the particles of the string and denote the tension by $N$ and stretch by $\nu$.  We denote the position of the particle $s$ along the $\bs i$ axis at time $t$ by $\chi(s,t)$ so that the stretch $\nu(s,t) = \chi_s(s,t)$.  The variables $\chi_s$ and $N$ are related via
\begin{align}
	\chi_s(s,t) =
	\begin{cases} 
		\frac{1}{E}N(s,t) &\mbox{ if } 0 < N(s,t) \leq N_1, \\
				\nu_1 &\mbox{ if } N(s,t) \geq N_1,
	\end{cases}\label{eq:chiN_relation}
\end{align}
where $\nu_1 > 1$ is the maximal stretch of the string, $N_1 > 0$ is the threshold tensile force resulting in inextensibility and $E = N_1/\nu_1$.
The equations of motion reduce to the simple balance law: 
\begin{align}
\begin{split}
&\gamma \chi_{tt}(s,t) = N_s(s,t), \quad (s,t) \in [0,\infty) \times [0,T], \\
&\chi(0,t) = 0, \quad N(\infty,t) = \zeta t + \tau, 
\end{split}\label{eq:1d_cons}
\end{align}
where $\gamma > 0$ is the mass density and $\zeta \geq 0$, $\tau \in (0,N_1)$ determine the prescribed tension at the infinite end. 
When the tension is a function of stretch, $N(s,t) = \hat N(\chi_s(s,t))$, as in the classical treatment of elastic strings, \eqref{eq:1d_cons} falls within the well-studied realm of hyperbolic balance laws of one spatial variable (see e.g. \cite{Dafermos_HypCons16}). However, the balance law \eqref{eq:1d_cons} with relation between the state variables given by \eqref{eq:chiN_relation} is, to the author's knowledge, novel. We assume that the string for each time $t$ contains an inextensible segment parameterized by $s \in [0,\sigma(t)]$ where the tension $N(\cdot,t) \in C^1([0,\sigma(t)])$ exceeds the threshold for inextensibility $$N(s,t) > N_1 > 0, \quad s \in [0,\sigma(t)],$$ and an extensible segment parameterized by $s \in [\sigma(t),\infty)$ where $N(\cdot,t) \in C^1([\sigma(t),\infty))$ and $$0 < N(s,t) = E \chi_s(s,t) < N_1, \quad s \in [\sigma(t),\infty).$$ We further assume the (generic) condition 
\begin{align}
N(\sigma^-(t),t) > N_1 > N(\sigma^+(t),t) \label{eq:jump1}
\end{align}
holds. Then the point of transition $s = \sigma(t)$ between segments is a shock front, and therefore, these motions do dissipate energy due to entropy production (see Proposition \ref{p:diss}) in contrast to motions which have continuous state variables (see Section 2).

The organization and main results of this work are as follows. In Section 2 we briefly review the general theory of mechanical strings, give the precise formulation of the problem we consider and derive the reduced equations of longitudinal motion. In Section 3 we define the class of classical shock solutions we study and prove the local-in-time existence and uniqueness of shock solutions to \eqref{eq:1d_cons} for classical shock front initial data $(\chi, \p_t \chi, N) |_{t = 0}$ (see Theorem \ref{p:wellposedness}).  The boundary condition at $s = 0$ and the assumption of longitudinal motion effectively reduces the proof to solving a linear wave equation for $\chi$ in the extensible region $\{(s,t) : t \in [0,T], s \in [\sigma(t),\infty)\}$ with $\chi(\sigma(t),t) = \nu_1 \sigma(t)$. The condition \eqref{eq:jump1} implies the shock speed $\sigma'$ is greater than the speed of propagation for this wave equation (a version of \emph{Lax's entropy inequality}), and the result is then proved via the implicit function theorem and finite speed of propagation arguments. Such a simplification is not present for the full equations of motion, and the local-in-time existence and uniqueness of solutions outside of longitudinal motion is an open question. Finally, in Section 4 we prove the orbital asymptotic stability of an explicit family of piece-wise constant stretched motions, 
\begin{align}
	\chi_s(s,t) = 
	\begin{cases}
		\nu_1 &\mbox{ if } s \in [0,\sigma(t)) \\
		\frac{\zeta t + \tau}{E} &\mbox{ if } s \in (\sigma(t),\infty),
	\end{cases}
\end{align}
parameterized by the initial shock front $\sigma(0)$ and shock speed $\sigma'(0)$ (see Theorem \ref{t:asymptotic_stability} and Theorem \ref{t:asymptotic_stability_poszeta}). When $\zeta > 0$, the final time of existence for this family and its perturbations is $T(\zeta) = (N_1 -\tau)/\zeta < \infty$, the time when the tension at infinity exceeds the threshold for inextensibility, $N(\infty,T(\zeta)) = N_1$. Since the tension at the end $s = \infty$ is growing to the threshold value $N_1$ when $\zeta >  0$, stability in this case is more delicate than $\zeta = 0$. One must show the stretch in the extensible segment grows in a controlled way like $(\zeta t + \tau)/E$ to leading order so that a second inextensible segment does not form. This requires a careful analysis of the relationship between the shock front $\sigma$ and the size of the perturbation in the extensible region $[\sigma(t),\infty)$.  For $\zeta > 0$ our results prove the existence of an open set of initial data launching motions having one inextensible segment and one extensible segment for $t < T(\zeta)$ and becoming fully inextensible $(\sigma(t) \rar \infty)$ as $t \rar T(\zeta)$. Moreover, these motions result in the blow-up of the tension: for all $s \in [0,\infty)$, 
\begin{align}
\lim_{t \rar T(\zeta)^-} N(s,t) = \infty. 
\end{align}
Whether or not the family of piece-wise constant stretched motions are orbitally asymptotically stable outside of purely longitudinal motion is an open question. 

We remark that our main results still hold with minor modifications if instead of assuming a linear relation between stretch and tension \eqref{eq:chiN_relation} we assume a linear relation between \emph{strain} $e := \nu - 1$ and tension, i.e. Hooke's law. In this case, \eqref{eq:chiN_relation} is replaced with 
\begin{align}
	\chi_s(s,t) - 1 =
	\begin{cases} 
		\frac{1}{E}N(s,t) &\mbox{ if } 0 < N(s,t) \leq N_1, \\
		\nu_1 -1&\mbox{ if } N(s,t) \geq N_1,
	\end{cases}\label{eq:chiN_relation2}
\end{align}
where $E := N_1 / (\nu_1 - 1)$ is \emph{Young's modulus} for the string. 
We expect our main results can be generalized for longitudinal motion involving more general stretch-limiting constitutive relations than \eqref{eq:chiN_relation} or \eqref{eq:chiN_relation2} with the possible exception of Theorem \ref{t:asymptotic_stability} due to a potential second shock forming in the extensible segment for small perturbations.  

\section{Formulation of the problem}

In this section we review of the general theory of mechanical strings, state the precise assumptions for the problem we consider and derive the reduced equations of motion.  

\subsection{Mechanical strings}

A brief overview of mechanical strings is as follows (see \cite{Antman79, AntmanBook} for a more extensive exposition). 
Let $\{ \bs i, \bs j, \bs k \}$ be a fixed right-handed orthonormal basis for Euclidean space $\bbE^3$. Let $I \subset \bbR$ be a fixed reference interval which parameterizes the \emph{particles} or \emph{material points} of the string. The \emph{configuration} of the string at time $t$ is the curve $I \ni s \mapsto \bs r(s,t) \in \bbE^3$, and the tangent to the curve $\bs r(\cdot, t)$ at $s$ is $\bs r_s(s,t) = \p_s \bs r(s,t)$. The \emph{stretch} $\nu(s,t)$ of the string at $(s,t)$ is 
\begin{align*}
	\nu(s,t) := |\bs r_s(s,t)|.
\end{align*}
We require that the configuration is regular so that stretch is always positive throughout the motion of the string. As a result of balance of linear momentum, the general \emph{classical equations of motion} for a string are given by: 
\begin{align}
	(\rho A)(s) \bs r_{tt}(s,t) = \bs n_s(s,t) + \bs f(s,t), \quad (s,t) \in I \times [0,T]. \label{eq:motion}
\end{align}
Here $(\rho A)(s)$ is the mass per unit reference length at $s$, $\bs f(s,t)$ is the body force per unit reference length at $(s,t)$, and $\bs n(s,t)$ is the contact force at $(s,t)$. 
A defining characteristic of a string is that the contact force is tangential to the configuration of the string: there exists a scalar-valued function $N(s,t)$, the \emph{tension}, such that 
\begin{align*}
	\bs n(s,t) = N(s,t) \frac{\bs r_s(s,t)}{|\bs r_s(s,t)|}.
\end{align*}

The mechanical properties of a string are modeled by specifying a relation between the stretch $\nu$ and tension $N$. Classically, a string is said to be \emph{elastic} if there exists a function $\hat N(\nu,s)$ such that 
\begin{align}
	N(s,t) = \hat N(\nu(s,t), s), \label{eq:constN}
\end{align}
with $\hat N(1,s) = 0$, $\nu \mapsto \hat N(\nu,s)$ increasing, $\lim_{\nu \rar 0} \hat N(\nu,s) = -\infty$ and $\lim_{\nu \rightarrow \infty} \hat N(\nu,s) = \infty$ for all $s$. These mathematical restrictions on $\hat N$ reflect the physically reasonable assumptions that an un-stretched configuration is not in a state of tension, an increase in tension leads to an increase in stretch, a state of zero stretch requires infinite compression force and a state of infinite stretch requires infinite tensile force.  Thus, $\hat N(\cdot,s)$ has an inverse function $\hat \nu(\cdot, s)$, and the constitutive relation between stretch and tension takes the form 
\begin{align}
	\nu(s,t) = \hat \nu(N(s,t),s). \label{eq:const}
\end{align}

\subsection{Longitudinal motion for stretch-limited elastic strings}

In this paper we study the dynamics of a simple class of uniform, semi-infinite stretch-limited elastic strings introduced in \cite{Rod20}. The constitutive relation can be expressed via \eqref{eq:const} but not via \eqref{eq:constN}, and we assume that the string becomes inextensible once a threshold value of tensile force is reached.  Mathematically, the string satisfies:  
\begin{itemize}
\item for all $(s,t) \in [0,\infty) \times [0,T]$, 
\begin{align}
(\rho A)(s) = \gamma > 0, \quad \nu(s,t) = \hat \nu(N(s,t)), 
\end{align}
where $\hat \nu : \bbR \rar (0,\infty)$ is absolutely continuous, nondecreasing and 
$\hat \nu(0) = 1$,  
\item there exist $N_1 > 0$ and $\nu_1 > 1$ such that $\hat \nu \in C^\infty([0, N_1];[1, \nu_1])$ is strictly increasing and for all $N \geq N_1$,
\begin{align}
\hat \nu(N) = \nu_1. 
\end{align} 
For simplicity we assume the constitutive relation is linear for all values of tension near the threshold: there exists $\eta > 0$ such that for all $N \in [\eta,N_1]$
\begin{align}
\hat \nu(N) = \frac{1}{E} N, \label{eq:hooke}
\end{align}
where $E := N_1/ \nu_1$. As mentioned in the introduction, our main results hold with minor modifications if instead of \eqref{eq:hooke} we assume a linear relation between the \emph{strain} $e := \nu - 1$ and the tension (Hooke's law):  there exists $\eta \geq 0$ such that for all $N \in [\eta,N_1]$
\begin{align}
	e = \hat e(N) = \frac{1}{E} N, \label{eq:hooke2}
\end{align} 
where $E := N_1 /(\nu-1)$ is \emph{Young's modulus} for the string. For definiteness, our statements and proofs will be in the setting of the first relation \eqref{eq:hooke}. 
\end{itemize}

Strings satisfying the previous assumptions admit a stored energy functional for motions satisfying $N > \eta$ and may be described as \emph{elastic}: there is no dissipation for motions that are sufficiently smooth. Indeed, define the stored energy functional via 
\begin{align}
	W(N, \nu) = 1_{\{ N \leq N_1 \}} \frac{E}{2} \nu^2 + 
	1_{\{N \geq N_1 \}} \frac{E}{2} \nu_1^2, 
\end{align}
the total stored energy of the material segment $[a,b]$, 
\begin{align}
	E(t) := \int_a^b W(N(s,t), \nu(s,t)) ds. 
\end{align}
and the kinetic energy of the material segment $[a,b]$, 
\begin{align}
K(t) := \int_a^b \frac{\gamma}{2} |\bs r_t(s,t)|^2 ds.
\end{align}
Then as long as the state variables $(\bs r_t, \bs r_s, \bs n) \in C^1([a,b] \times [0,T]; \bbR^9)$ and $N > \eta$ during the motion, an easy calculation using \eqref{eq:motion} and \eqref{eq:hooke} shows that for all $t \in [0,T]$
\begin{align}
	K'(t) + E'(t) = P(t) := \bs n(b,t) \cdot \bs r_t(b,t) - \bs n(a,t) \cdot \bs r_t(a,t) \label{eq:energy} 
\end{align} 
where $P(t)$ is the power exerted on the segment $[a,b]$ by the tensile forces in the string. By the first law of thermodynamics \eqref{eq:energy} implies that no mechanical work is converted to heat. 

In this work we specialize our study to the longitudinal motion for a uniform, tensile, stretch-limited string under no external body force, fixed at the end $s = 0$, pulled taut at the end at $s = \infty$ and containing an inextensible segment and an extensible segment: 
\begin{itemize}
	\item for all $(s,t)$,
	\begin{align}
	\bs r(s,t) = \chi(s,t) \bs i, \quad \chi_s(s,t) > 0, 
	\end{align}
	so 
	\begin{align}
	\nu(s,t) = \chi_s(s,t), \quad \bs n(s,t) = N(s,t) \bs i, \\
	\chi_s(s,t) = \begin{cases}
	\frac{1}{E} N(s,t) &\mbox{ if } N(s,t) \in [\eta, N_1], \\
	\nu_1 &\mbox{ if } N(s,t) \in [N_1, \infty),
	\end{cases}
	\end{align}
	
	\item there is no body force, $\bs f(s,t) = \bs 0$, 
	
	\item the boundary conditions 
	\begin{align}
	\bs r(0,t) = \bs 0, \quad \bs n(\infty,t) = (\zeta t + \tau) \bs i, \label{eq:bdy_cdns}
	\end{align}
	where $\zeta \geq 0$ and $\tau \in (\eta, N_1)$,
	
	\item there exists $\sigma \in C^2([0,T];(0,\infty))$ such that $\sigma' \geq 0$ and
	\begin{gather}
	\forall s \in [0,\sigma(t)), \quad N(s,t) \in (N_1,\infty),  \\ 
	\forall s \in (\sigma(t),\infty), \quad N(s,t) \in (\eta,N_1), \label{eq:N_conditions}
	\end{gather}

\item and the generic condition that for all $t \in [0,T]$, 
\begin{align}
	N(\sigma^-(t),t) > N_1 > N(\sigma^+(t),t). \label{eq:jump}
\end{align}

\end{itemize}
	We refer to $[0,\sigma(t)]$ as the inextensible (material) segment, 
$[\sigma(t),\infty)$ as the extensible (material) segment and the point $s = \sigma(t)$ as the shock front. We note that the assumption $\sigma' \geq 0$ reflects the physical assumption that the inextensible segment is \emph{growing} in time.

We consider weak solutions to the equations of motion \eqref{eq:motion} on $[0,\infty) \times [0,T]$ which are also satisfied classically on 
\begin{align}
\cl I_\sigma := \{(s,t) : t \in [0,T], 
s \in [0, \sigma(t)] \}
\end{align}
and
\begin{align}
\cl E_\sigma := \{(s,t) : t \in [0,T] , s \in [\sigma(t), \infty) \}
\end{align} 
separately. The property of $(\bs r, \bs n)$ being a weak solution is equivalent to the Rankine-Hugoniot jump conditions: 
\begin{align}
\llbracket \bs n \rrbracket + \sigma' \llbracket \bs r_t \rrbracket = \bs 0, \label{eq:RH_cond}
\end{align} 
where $\llbracket \bs y \rrbracket(\sigma(t),t) = \bs y(\sigma(t)^+,t) - \bs y(\sigma(t)^-,t)$ is the jump across the point $s = \sigma(t)$ at time $t$.

\subsection{Reduced equations of motion and jump conditions}

We now derive the reduced equations of motion from \eqref{eq:motion}, the Rankine-Hugoniot conditions \eqref{eq:RH_cond} and our assumptions from the previous subsection. 
For $(s,t) \in \cl E_\sigma$, where the string is extensible, \eqref{eq:motion} and the boundary condition \eqref{eq:bdy_cdns} reduce via \eqref{eq:hooke} to the wave equation 
\begin{align}
\begin{split}
&\chi_{tt}(s,t) = \frac{E}{\gamma} \chi_{ss}(s,t), \\
&\chi_s(\infty,t) = \frac{\zeta t + \tau}{E}, 
\end{split}\label{eq:ext_equations}
\end{align} 
For $(s,t) \in \cl I_\sigma$,
where the string is inextensible and $\chi_s(s,t) = \nu_1$, we have 
\begin{align}
\chi(s,t) = \int_0^s \nu_1 d\xi = \nu_1 s, 
\end{align}
and \eqref{eq:motion} reduces to $N_s = 0$ so,
\begin{align}
\begin{split}
N(s,t) &= N(\sigma^-(t),t) \\
\chi(s,t) &= \nu_1 s. 
\end{split}\label{eq:inext_equations1}
\end{align}
The Rankine-Hugoniot conditions \eqref{eq:RH_cond} reduce to 
\begin{align}
\llb N \rrb + \gamma \sigma'(t) \llb \chi_t \rrb = 0. \label{eq:RH_conditions} 
\end{align} 
Imposing the condition that $\chi(s,t)$ is continuous across $s = \sigma$ (for brevity we drop the dependence of $\sigma$ on $t$) implies  
\begin{align}
\chi(\sigma^+,t) = \nu_1 \sigma = \frac{N_1}{E} \sigma \label{eq:conty}
\end{align} 
and thus
\begin{align}
\chi_t(\sigma^+,t) &= \frac{\sigma'}{E}(N_1 - N(\sigma^+,t)). \label{eq:zt_jump}
\end{align}
Since $\chi_t(s,t) = 0$ for $s \in [0,\sigma]$, we conclude 
\begin{align}
\gamma \sigma' \llb \chi_t \rrb = \frac{\gamma (\sigma')^2}{E}(N_1 - N(\sigma^+,t)) \label{eq:kinetic_jump}
\end{align}
which by \eqref{eq:RH_conditions} yields 
\begin{align}
N(\sigma^-,t) = N(\sigma^+,t) + \frac{\gamma (\sigma')^2}{E}(N_1 - N(\sigma^+,t)). \label{eq:tension_jump}
\end{align}
Thus, in the extensible region $\cl I_{\sigma}$, 
\begin{align}
\begin{split}
N(s,t) &= N(\sigma^+,t) + \frac{\gamma (\sigma')^2}{E}(N_1 - N(\sigma^+,t)), \\
\chi(s,t) &= \nu_1 s. 
\end{split}\label{eq:inext_equations}
\end{align}
In summary, the equations of motion \eqref{eq:motion} and Rankine-Hugoniot conditions \eqref{eq:RH_cond} are equivalent to the equations \eqref{eq:ext_equations} in the region $\cl E_{\sigma}$, the equations \eqref{eq:inext_equations} in the region $\cl I_{\sigma}$, the condition \eqref{eq:conty} and our assumption \eqref{eq:jump}.

From \eqref{eq:jump} and \eqref{eq:inext_equations}, it follows that the shock speed is greater than the speed of propagation in the extensible region to the right of $\sigma(t)$:
\begin{align}
	\infty > \sigma'(t) > \sqrt{E/\gamma}. \label{eq:shock_speed_bounds}
\end{align} 
Since the speed of propagation of the inextensible region to the left of $\sigma(t)$ is infinite, \eqref{eq:shock_speed_bounds} reflects that the solutions we study satisfy a generalized version of \emph{Lax's entropy inequality} \cite{Lax_CPAM57}. By choosing appropriate units, we may assume that 
\begin{align}
\gamma = E = 1, 
\end{align}
so that $N_1 = \nu_1$ and for all $(s,t) \in \cl E_{\sigma}$
\begin{align}
N(s,t) = \chi_s(s,t). 
\end{align}

\section{Existence and Uniqueness of Shock Solutions}

We now define precisely the class of initial data and shock solutions to the equations of motion that we study in this work and prove our existence and uniqueness result.  

\subsection{Preliminaries} 

\begin{defn}\label{d:diss_data}
	We say a triple of real-valued functions 
	\begin{align}
	(\chi_0, \chi_1, N_0) 
	\in C([0,\infty)) \times L^\infty([0,\infty)) \times L^\infty([0,\infty)). 
	\end{align}
	 is a set of classical \emph{shock front initial data} if there exists a unique $\sigma_0 \in (0,\infty)$ (the initial shock front) such that 
	\begin{enumerate}
	\item $(\chi_0, \chi_1) \in C^2 \times C^1([\sigma_0,\infty))$ and $N_0 \in C^1([\sigma_0,\infty))$, 
	\item for all $s \in [\sigma_0, \infty)$, $N_0(s) = \chi_0'(s)$ and
	\begin{gather}
	\eta < N_0(s) < N_1, 
	\end{gather}
	\item $\chi_0'(\infty) = \tau, \quad \chi_1'(\infty) = \zeta, \quad \chi_0(\sigma_0) = N_1\sigma_0$, 
	\item $\sigma_1 := \frac{\chi_1(\sigma_0^+)}{N_1 - \chi_0'(\sigma_0^+)} > 1,$ 
	\item for all $s \in [0,\sigma_0]$, 
	\begin{align}
	\begin{split}
	\chi_0(s) &= N_1 s, \\
	\chi_1(s) &= 0, \\
	N_0(s) &= N_0(\sigma_0^+) + (\sigma_1)^2(N_1 - N_0(\sigma_0^+)) > N_1.
	\end{split}\label{eq:inextdata_definition}
	\end{align}
	\end{enumerate}
\end{defn}

We remark that a simple class of shock front initial data is given by 
\begin{align}
\chi_0(s) = N_1 \sigma_0 + \tau(s - \sigma_0), \quad s \in [\sigma_0,\infty). \label{eq:z0_g0_data} 
\end{align}
and \emph{any} $\chi_1 \in C^1([\sigma_0,\infty))$ satisfying 
\begin{align}
\chi_1'(\infty) = \zeta, \quad \chi_1(\sigma_0) > (N_1 - \tau). 
\end{align}
(This data $(\chi_0,\chi_1,N_0)|_{[\sigma_0, \infty)}$ is extended to $[0,\sigma_0]$ via \eqref{eq:inextdata_definition}.) 

\begin{defn}\label{d:diss_solution}
Let $(\chi_0,\chi_1,N_0)$ be a set of classical shock front initial data with initial shock front $\sigma_0 \in (0,\infty)$. We say a pair 
\begin{align}
(\chi,N) \in C([0,\infty) \times [0,T]) \times L^\infty([0,T] \times [0,\infty))
\end{align}
is a classical \emph{shock solution} on the time interval $[0,T]$ with initial data $(\chi_0,\chi_1,N_0)$ if
$\p_t \chi \in L^\infty([0,\infty) \times [0,T])$ and 
there exists a unique curve $$\sigma \in C^2([0,T];[\sigma_0,\infty)),$$ the shock front, with $\sigma(0) = \sigma_0$ such that 
\begin{enumerate}
	\item  $\chi \in C^2(\cl E_{\sigma})$ and $N \in C^1(\cl E_\sigma)$ where, as before, 
	\begin{align}
	\cl E_{\sigma} := \{(s,t) : t \in [0,T], s \in [\sigma(t),\infty) \},
	\end{align}
	\item on 
	$\cl E_{\sigma}$, $N(s,t) = \chi_s(s,t)$ and 
	\begin{align}
	\eta < N(s,t) < N_1, 
	\end{align} 
	\item $\chi(\sigma^+,t) = N_1 \sigma$, 
	\item $\sigma' = \frac{\chi_t(\sigma^+,t)}{N_1 - N(\sigma^+,t)} > 1$, 
	\item on $\cl I_{\sigma} := \{ (s,t) : t \in [0,T], s \in [0,\sigma]\}$,
	\begin{align}
	\begin{split}
	\chi(s,t) &= N_1 s, \\
	N(s,t) &= N(\sigma^+,t) + (\sigma')^2(N_1 - N(\sigma^+,t)) > N_1. 
	\end{split}\label{eq:inext_definition}
	\end{align} 
	\item on $\cl E_{\sigma}$, $\chi$ satisfies the wave equation
	\begin{align}
	\begin{split}
	&\chi_{tt}(s,t) = \chi_{ss}(s,t), \\
	&\chi_s(\infty,t) = \zeta t + \tau, \\
	&\chi(s,0) = \chi_0(s), \chi_t(s,0) = \chi_1(s), 
	\end{split} \label{eq:ext_equations1} 
	\end{align}
\end{enumerate} 
\end{defn}
		
\subsection{Existence and uniqueness of classical shock solutions}

We now prove the existence and uniqueness of classical shock solutions for given classical shock front initial data. 

\begin{thm}\label{p:wellposedness}
Suppose $(\chi_0,\chi_1,N_0)$ is a set of classical shock front initial data with initial shock front $\sigma_0 \in (0,\infty)$. 
Then there exist $T = \hat T(\chi_0,\chi_1) > 0$ and a unique classical shock solution $(\chi,N)$ to the equations of motion on the time interval $[0,T]$. 
\end{thm}

\begin{proof}
(Existence)
We first extend the initial data $$(\chi_0, \chi_1) \in C^2 \times C^1([\sigma_0,\infty))$$ to a pair $$(\hat \chi_0, \hat \chi_1) \in C^2 \times C^1(\bbR)$$ such that for all $s \in [\sigma_0 - \delta,\infty)$ 
\begin{align}
N_0 < \hat \chi_0'(s) < N_1.
\end{align}
Let $\hat \chi \in C^2(\bbR^2_{s,t})$ be the unique solution to the wave equation 
\begin{align}
&\hat \chi_{tt}(s,t) = \hat \chi_{ss}(s,t), \quad (s,t) \in \bbR^2_{s,t}, \\
&\hat \chi(s,0) = \hat \chi_0(s), \hat \chi_t(s,0) = \hat \chi_1(s), \quad s \in \bbR. 
\end{align}
By D'Alembert's formula
\begin{align}
\hat \chi(s,t) = \frac{1}{2} \Bigl (
\hat \chi_0(s+t) + \hat \chi_0(s-t)
\Bigr ) + \frac{1}{2} \int_{s-t}^{s+t} \hat \chi_1(\xi) d\xi \label{eq:dalembert}
\end{align}
and the assumptions on the data, we have for all $t \in \bbR$,
\begin{align}
	\hat \chi_s(\infty,t) = \zeta t + \tau.
\end{align} 
By continuity and choosing $T_0 > 0$ sufficiently small, we have for all $(s,t) \in [\sigma_0-\delta,\infty) \times [-T_0,T_0]$, 
\begin{align}
\eta < \hat \chi_s(s,t) < N_1. 
\end{align}

Consider the function
\begin{align}
f(\sigma, t) = N_1 \sigma - \hat \chi(\sigma,t), 
\quad  (\sigma,t) \in \bbR \times [-T_0,T_0].  
\end{align}
Then $f(\sigma_0,0) = N_1 \sigma_0 - \chi_0(\sigma_0) = 0$, and 
\begin{align}
f_{\sigma}(\sigma_0,0) = {N_1} - \chi_0'(\sigma_0^+) = 
N_1 - N_0(\sigma_0^+)> 0. 
\end{align}
By the implicit function theorem, there exists $T_1 \in (0,T_2)$, $\eps > 0$ such that for all $t \in [-T_1,T_1]$ there exists a unique $\hat \sigma(t) \in [\sigma_0 - \eps, \sigma_0 + \eps]$ such that $f(\hat \sigma(t),t) = 0$. Moreover, by differentiating the equation $f(\hat \sigma(t),t) = 0$ in time, we have 
\begin{align}
\hat \sigma'(t) = \frac{\hat \chi_t(\hat \sigma(t),t)}{N_1 - \hat \chi_s(\hat \sigma(t),t)} \label{eq:sigma_der}
\end{align}
which implies $\hat \sigma'(t)$ is $C^1$ and 
\begin{align}
\hat \sigma'(0) = \sigma_1 = \frac{\chi_1(\sigma_0^+)}{N_1 - N_0(\sigma_0^+)} > 1. 
\end{align}
By choosing $\hat T < T_1$ sufficiently small, we can ensure that for all $t \in [0,\hat T]$
\begin{align}
\hat \chi(\hat \sigma(t),t) = {N_1}\hat \sigma(t), \quad 
\hat \sigma'(t) > 1. 
\end{align}
The latter inequality along with $N_1 > \hat \chi_s(\hat \sigma(t),t)$ implies 
\begin{align}
\hat \chi_s(\hat \sigma(t),t) + {(\hat \sigma'(t))^2}(N_1 - \hat \chi_s(\hat \sigma(t),t))
> N_1.
\end{align}
Defining $T := \hat T$, $\sigma := \hat \sigma$, $(\chi,N)|_{\cl E_\sigma} := (\hat \chi,\hat \chi_s)|_{\cl E_{\sigma}}$ and $(\chi,N)|_{\cl I_\sigma}$ via \eqref{eq:inext_definition} we obtain a shock solution on $[0,T]$ with initial data $(\chi_0,\chi_1,N_0)$. 

(Uniqueness) Suppose $(\check \chi, \check N)$ is another  shock solution on $[0,T]$ with initial data $(\chi_0,\chi_1,N_0)$ and initial shock front $\sigma_0$. Let $\bar \chi$ be the unique $C^2$ solution to the wave equation \eqref{eq:ext_equations1} on $\{(s,t): t \in [0,T], s \in [\sigma_0 + t,\infty)\}$. By finite speed of propagation, $\bar \chi$ is a $C^2$ extension of $\chi |_{\cl E_\sigma}$ and $\check \chi |_{\cl E_{\check \sigma}}$, and thus, $\sigma$ and $\check \sigma$ solve the ordinary differential equation
\begin{align}
\bar \sigma' = \frac{\bar \chi_t(\bar \sigma,t)}{N_1 - \bar \chi_s(\bar \sigma,t)}, \quad \bar \sigma(0) = \sigma_0,
\end{align}
which implies $\sigma = \check \sigma$. We conclude $\chi |_{\cl E_\sigma} = \bar \chi |_{\cl E_{\sigma}} = \check \chi|_{\cl E_{\sigma}}$, proving uniqueness. 
\end{proof}

We note that our uniqueness proof does not rely on $T$ being the same as that from the existence proof. Therefore, we define the \emph{final time of existence} $T_+$ for a shock solution $(\chi,N)$ to be the largest time such that $(\chi,N)$ is a shock solution on $[0,T_+)$. We refer to $[0,T_+)$ as the \emph{maximal interval of existence} for $(\chi,N)$.

By finite speed of propagation we have the following simple observation that we can identify a classical shock solution $(\chi,N)$ on $\cl E_{\sigma}$ with the restriction of a solution to the wave equation defined on the strictly larger domain $\{(s,t) : t \in [0,T], s \in [\sigma(0)+t, \infty) \}$.  

\begin{prop}\label{p:cont_criterion}
Suppose that $(\chi,N)$ is a classical shock solution on a time interval $[0,T)$ with shock front $\sigma$ and initial data $(\chi_0, \chi_1, N_0)$. Let 
$\bar \chi$ be the unique $C^2$ solution to the wave equation \eqref{eq:ext_equations1} on 
$\{(s,t): t \in [0,\infty), s \in [\sigma(0) + t,\infty)\}$. Then 
\begin{align}
	\cl E_{\sigma} \subset \{(s,t): t \in [0,\infty), s \in [\sigma(0) + t,\infty)\},
\end{align}
and $\chi_{\cl E_\sigma} = \bar \chi_{\cl E_{\sigma}}$. Consequentially, if 
\begin{align}
\sigma(T) := \lim_{t \rar T} \sigma(t) < \infty, \label{eq:ext_assumption}
\end{align}
and, 
\begin{align}
&\eta < \bar \chi_s(s,T) < N_1, \quad s \in [\sigma(T),\infty), \\
& \frac{\bar \chi_t(\sigma(T),T)}{N_1 - \bar \chi_s (\sigma(T),T)} > 1, 
\end{align}
then $(\chi,N)$ extends uniquely to a  shock solution on $[0,\bar T]$ with 
$\bar T > T$.   
\end{prop}  

\subsection{Dissipation}

We conclude this section by showing that these motions dissipate the mechanical energy of material segments. 

\begin{prop}\label{p:diss}
	Suppose $(\chi,N)$ is a  shock solution to the equations of motion on the time interval $[0,T]$ with shock front $\sigma$, and suppose that $a,b \in (0,\infty)$ with $a < \sigma(t) < b$.  Define the total energy (the sum of the total kinetic energy and total stored energy of the string) of the material segment $[a,b]$ via  
	\begin{align}
		K(t) + E(t) &:= \int_\sigma^b \frac{1}{2} (\chi_t(s,t))^2 ds + \int_a^\sigma \frac{1}{2} N_1^2 ds 
		+ \int_\sigma^b \frac{1}{2} (N(s,t))^2 ds. 
	\end{align}
	Then for all $t \in [0,T]$, $$K'(t) + E'(t) = P(t) + Q(t)$$ where $P(t)$ is the total power exerted by the tensile forces on the segment $[a,b]$, 
	\begin{align}
		P(t) := N(b,t) \chi_t(b,t) - N(a,t) \chi_t(a,t) = 
		N(b,t) \chi_t(b,t), 
	\end{align}
	and $Q(t)$ is the total heat power,
	\begin{align}
		Q(t) := -\frac{\sigma'}{2}(N_1 - N(\sigma^+,t))^2((\sigma')^2-1) < 0.
	\end{align} 
\end{prop}

\begin{proof}
	Using the relation $N = \chi_s$ on $\cl E_\sigma$, the equations of motion \eqref{eq:ext_equations1} and integration by parts we deduce 
	\begin{align}
		K' + E' = P + Q
	\end{align}
	where 
	\begin{align}
		-Q(t) = 
		\chi_t(\sigma^+,t) N(\sigma^+,t) + 
		\frac{\sigma'}{2}
		(\chi_t(\sigma^+,t))^2 + \frac{\sigma'}{2} (N(\sigma^+,t)^2 - N_1^2).
	\end{align}
	By the Rankine-Hugoniot conditions and the relation $$\llb \chi_t \rrb(\sigma^+,t) = \chi_t(\sigma^+,t) = \sigma' (N_1 - N(\sigma^+,t))$$ we deduce 
	\begin{align}
		\chi_t(\sigma^+,t) N(\sigma^+,t) &+ \frac{1}{2} \sigma' (\chi_t(\sigma^+,t))^2 \\ &=
		\llb \chi_t \rrb(\sigma^+,t) \left (
		N(\sigma^+,t) + \frac{1}{2} \sigma' \llb \chi_t \rrb(\sigma^+,t)
		\right) \\
		&= \llb \chi_t \rrb(\sigma^+,t) \frac{N(\sigma^+,t) + N(\sigma^-,t)}{2} \\
		&= \frac{\sigma'}{2}(N_1 - N(\sigma^+,t))(N(\sigma^+,t) + N(\sigma^-,t)).
	\end{align}
	By Definition \ref{d:diss_solution} we conclude 
	\begin{align}
		-Q(t) &= 
		\frac{\sigma'}{2} \Bigl (
		(N_1 - N(\sigma^+,t))(N(\sigma^+,t) + N(\sigma^-,t)) \\
		&\qquad- (N_1^2 - N(\sigma^+,t)^2)
		\Bigr ) \\
		&= \frac{\sigma'}{2}(N_1 - N(\sigma+,t))(N(\sigma^-,t) - N_1) \\
		&= \frac{\sigma'}{2}(N_1 - N(\sigma^+,t))^2((\sigma')^2-1) > 0
	\end{align}
	as desired. 
\end{proof}

\section{Piece-wise constant stretched motions}
In this section we prove the orbital asymptotic stability of an explicit two-parameter family of piece-wise constant stretched motions. 
\subsection{The two-parameter family of piece-wise constant stretched motions}
We define 
\begin{align}
	T(\zeta) := 
	\begin{cases}
		\infty  &\mbox{ if } \zeta = 0, \\
		\frac{N_1 - \tau}{\zeta} &\mbox{ if } 
		\zeta > 0.
	\end{cases}
\end{align}
When $\zeta > 0$, this is precisely the time when the tension at $s = \infty$ surpasses the threshold tensile value, $N_1$, for inextensibility. In particular, the solutions to the equations of motion having one inextensible segment and one extensible segment satisfy $T_+ \leq T(\zeta)$.

We first exhibit the solutions which have piece-wise constant stretch, 
\begin{align}
\nu(s,t) = \chi_{s}(s,t) = \begin{cases}
N_1 &\mbox{ if } s\in [0,\sigma(t)], \\
\zeta t + \tau &\mbox{ if } s \in [\sigma(t),\infty),
\end{cases} \label{eq:const_str}
\end{align}
where $\sigma$ is the shock front separating the inextensible segment $[0, \sigma]$ from the extensible segment $[\sigma, \infty)$. Via simple calculations, we have the following.

\begin{prop}\label{p:solutions_g0}
Let $(\sigma_0, \sigma_1) \in (0,\infty) \times (1,\infty)$.
Define a shock front $\sigma$ and pair of functions for $(s,t) \in \cl E_\sigma$ by 
\begin{align}
\begin{split}
\sigma(t ; \zeta, \sigma_0,\sigma_1) &= \sigma_0 + \frac{\sigma_1(N_1 - \tau) t}{N_1 - (\zeta t +\tau)}, \\
\chi(s,t ; \zeta, \sigma_0,\sigma_1) &= N_1 \sigma_0 + (\zeta t + \tau)(s- \sigma_0) + \sigma_1 (N_1 - \tau)t, \\
N(s,t; \zeta, \sigma_0, \sigma_1) &= \zeta t + \tau,
\end{split}\label{eq:solutions_g0}
\end{align}
extended to $\cl I_{\sigma}$ via \eqref{eq:inext_definition}.  Then \eqref{eq:solutions_g0} defines the unique solution to the equations of motion on the maximal interval of existence $[0,T(\zeta))$ satisfying \eqref{eq:const_str}, $\sigma(0) = \sigma_0$ and $\sigma'(0) = \sigma_1$.    
\end{prop}

The values of the following state variables corresponding to a piece-wise constant stretched motion with parameters $(\sigma_0,\sigma_1) \in (0,\infty) \times (1,\infty)$ (and omitting the dependence of the solution on $\zeta, \sigma_0, \sigma_1$) are given by: for $s \in [\sigma(t), \infty)$
\begin{align}
\chi_{s}(s,t) = \zeta t + \tau, \quad \chi_t(s,t) = 
\zeta(s - \sigma_0) + \sigma_1(N_1 - \tau), \\
\sigma'(t) = \sigma_1 \frac{(N_1 - \tau)^2}{(N_1 - (\zeta t + \tau))^2}
= \frac{\zeta(\sigma(t) - \sigma_0) + \sigma_1(N_1 - \tau)}{N_1 - (\zeta t + \tau)},
\end{align}
and for $s \in [0,\sigma(t)]$ 
\begin{align}
N(s,t) = \zeta t + \tau + \frac{[\zeta(\sigma(t) - \sigma_0) + \sigma_1(N_1 - \tau)]^2}{N_1 - (\zeta t + \tau)}. 
\end{align}

\subsection{Perturbations of piece-wise constant stretched motions} 

We first note that the piece-wise constant stretched motions described in the previous subsection are, in general, \emph{unstable} to small $L^\infty$ perturbations of the state variables.   
Indeed, suppose $\zeta > 0$ and $(\sigma_0, \sigma_1) \in (0,\infty) \times (1, \infty)$. Then for all $\epsilon > 0$ and $R > \sigma_0$ one can construct initial data $(\chi_0, \chi_1, N)$ with initial shock front $\sigma_0$ such that  
\begin{align}
	\tau \leq \chi_0'(s) &\leq \tau + \epsilon, \quad s \in [\sigma_0, \infty), \\
	\chi_0'(s) &= \begin{cases}
		\tau &\mbox{ if } s \in [\sigma_0, R] \cup [4R,\infty) \\
		\tau + \eps &\mbox{ if } s \in [2R, 3R]
	\end{cases}, \\
	\chi_1(s) &= \zeta(s - \sigma_0) + \sigma_1(N_1 - \tau), \quad s \in [\sigma_0, \infty).
\end{align}
Then
\begin{align}
	\| \chi_0' - \tau \|_{L^\infty[\sigma_0, \infty)} + \| \chi_1 - [\zeta(s - \sigma_0) + \sigma_1(N_1 - \tau)] \|_{L^\infty} \leq \epsilon. 
\end{align}
By choosing $\eps$ sufficiently small, $R$ sufficiently large and using finite speed of propagation, we can then conclude using Proposition \ref{p:cont_criterion} that 
\begin{align}
	T_+ = \frac{N_1 - (\tau + \eps)}{\zeta} < T(\zeta),
\end{align} 
and for $t \in [0,T_+]$, $s \in [2R+t, 3R-t]$, 
\begin{align}
\chi_s(s,t) = \zeta t + \tau + \eps,
\end{align}
i.e. the string has formed a second inextensible material segment before the time $T(\zeta)$. The previous argument also previews why proving stability of piece-wise constant stretch states is more delicate when $\zeta > 0$: the increasing tension at $s = \infty$ and the nature of the perturbation may result in a second inextensible segment. 

We now prove piece-wise constant stretch motions are orbitally asymptotically stable under small \emph{weighted} $L^\infty$ perturbations of the state variables and their derivatives. We begin with the simpler case $\zeta = 0$.

\begin{thm}\label{t:asymptotic_stability}
	Assume $\zeta = 0$. Let $(\sigma_0, \sigma_1) \in (0,\infty) \times (1,\infty)$.  There exist $\eps_0  > 0$ and $C > 0$ such that for all $\eps < \eps_0$, the following is true. Suppose  
	$(\chi_0,\chi_1,N_0)$ is a set of initial data with initial shock front $\sigma_0$ such that there exists $r > 0$ such that 
	\begin{align}
	B:= \sup_{s \in [\sigma_0,\infty)} \Bigl [	| &\chi_0'(s) - \tau | + s^{r+1}|\chi_0''(s)| \\
	 &+
		| \chi_1(s) - \sigma_1(N_1 - \tau) | + 
		s^{r+1} |\chi_1'(s)| \Bigr ]
		 < \eps. \label{eq:pertubation_g0}
	\end{align}
Then the unique solution $(\chi,N)$ to the equations of motion with initial data $(\chi_0, \chi_1, N_0)$ satisfies $T_+ = T(0) = \infty$ and 
		 the state variables and shock speed remain close to and asymptotically approach those of a piece-wise constant stretched motion as $t \rar \infty$: there exist $\sigma_1^\infty \in (1,\infty)$ such that $\bigl |\frac{\sigma_1^\infty}{\sigma_1} - 1\bigr | \leq C \eps$ and for all $t \in [0,\infty)$ 
		\begin{align}
			\begin{split}
				\sup_{s \in [\sigma(t), \infty)}& s^{r+1} [ |\chi_{ss}(s,t)| + |\chi_{st}(s,t)| ] \leq C B, \\
				\sup_{s \in [\sigma(t), \infty)}& \Bigl [ |\chi_s(s,t) - \tau| +
				+ |\chi_t(s,t) - \sigma_1^\infty (N_1 - \tau)|  \Bigr ] \leq C B (1 +t )^{-r}, \\
				\Bigl |&\frac{\sigma'(t)}{\sigma_1^\infty} - 1 \Bigr | \leq C B (1+t)^{-r}, \\
				\sup_{s \in [0,\sigma(t)]} & \Bigl |N(s,t)([\tau + (\sigma_1^\infty)^2(N_1 - \tau)])^{-1} - 
				1 \Bigr |
				\leq C B (1 + t)^{-r}.
			\end{split}\label{eq:asymptotic}
		\end{align}
\end{thm}

\begin{proof}
We first note that by \eqref{eq:pertubation_g0} and the fundamental theorem of calculus,
\begin{align}
	|\chi_0'(s) - \tau | = \left |
	\int_s^\infty \chi_0''(\xi) d\xi 
	\right | \leq \frac{B}{r} s^{-r},
\end{align}
$\sigma_1^\infty := \frac{1}{N_1 -\tau}\lim_{s \rar \infty} \chi_1(s)$ exists
and 
\begin{align}
	|\chi_1(s) - \sigma_1^\infty(N_1 - \tau)| \leq \frac{B}{r} s^{-r},
\end{align}
for all $s \in [\sigma_0, \infty)$. In particular, by \eqref{eq:pertubation_g0}, 
\begin{align}
	|\sigma_1(N_1 - \tau) - \sigma_1^\infty(N_1 - \tau)| = |\sigma_1(N_1 - \tau) - \chi_1(\infty)| \leq \eps,
\end{align}
proving $\sigma_1^\infty \in (1,\infty)$ and 
\begin{align}
\Bigl | \frac{\sigma_1^\infty}{\sigma_1} - 1  \Bigr | \leq \frac{\eps}{N_1 - \tau}
\label{eq:sigmainfinity_bound}
\end{align}
for all $\epsilon$ sufficiently small. 	
	
Let $(\chi,N)$ be the unique solution to the equations of motion with shock front $\sigma$ and initial data $(\chi_0, \chi_1, N_0)$ defined on a maximal time interval of existence $[0,T_+)$.
Let $\bar \chi$ be the unique $C^2$ function on $\cl D_{\sigma_0} := \{(s,t) : t \in [0,\infty), s \in [\sigma_0 + t, \infty)\}$ solving the wave equation
\begin{align}
\begin{split}
&\bar \chi_{tt}(s,t) = \bar \chi_{ss}(s,t),  \\
&\bar \chi_s(\infty,t) = \tau, 
\end{split}\label{eq:wave_equation}\\
&\bar \chi(s,0) = \chi_0(s), \bar \chi_t(s,0) = \chi_1(s).
\end{align}
By D'Alembert's formula: 
\begin{align}
	\bar \chi(s,t) = \frac{1}{2} \Bigl (
	\chi_0(s+t) + \chi_0(s-t)
	\Bigr ) + \frac{1}{2} \int_{s-t}^{s+t} \chi_1(\xi) d\xi. \label{eq:dalembertbar}
\end{align}
By finite speed of propagation, 
\begin{align}
\chi |_{\cl E_{\sigma}} = \bar \chi |_{\cl E_{\sigma}}, \label{eq:finite_speed} 
\end{align}
and thus, the shock front $\sigma$ is the unique $C^2$ solution to 
\begin{align}
\sigma'(t) = \frac{\bar \chi_t(\sigma,t)}{N_1 - \bar \chi_s(\sigma,t)}, \quad \sigma(0) = \sigma_0. \label{eq:ode}
\end{align} 

By D'Alembert's formula we have for all $t \in [0,\infty), s \in [\sigma(t), \infty)$ 
\begin{align}
	|\bar \chi_s(s,t)& - \tau| \\ &=
	\Bigl |
	\frac{1}{2} (\chi'_0(s+t) - \tau) + \frac{1}{2} (\chi'_0(s+t) - \tau) \\
	&\qquad + \frac{1}{2} (\chi_1(s+t) - \sigma_1^\infty(N_1 - \tau)) - 
	\frac{1}{2}(\chi_1(s-t) - \sigma_1^\infty(N_1 - \tau))
	\Bigr | \\
	&\leq \frac{2B}{r} (\sigma - t)^{-r}, \label{eq:chis_estimate}
\end{align}
and similarly 
\begin{align}
	|\bar \chi_t(s,t) - \sigma_1^\infty(N_1 - \tau)| 
	&\leq \frac{2B}{r} (\sigma - t)^{-r}, \label{eq:chit_estimate} \\
	s^{r+1}[|\bar \chi_{ss}(s,t)| + |\bar \chi_{st}(s,t)|] &\leq B ( 1 + 
	(s-t)^{-r-1}s^{r+1} ).
	\label{eq:chiss_estimate} 
\end{align}

Since $\sigma - t \geq \sigma_0$, we conclude from \eqref{eq:pertubation_g0}, \eqref{eq:chis_estimate} and \eqref{eq:chit_estimate} that there exists $C_1 > 0$ such that for all $\eps$ sufficiently small, for all $(s,t) \in \cl D_{\sigma_0}$
\begin{align}
\eta < \tau - C_1\eps < \bar \chi_s(s,t) < \tau + C_1\eps < N_1, \\
\sigma_1^\infty(1 - C_1\eps) < 
\frac{\bar \chi_s(\sigma,t)}{N_1 - \bar \chi_s(\sigma,t)} < \sigma_1^\infty (1 + C_1 \eps). \label{eq:sigmadot_bound1}
\end{align}
By Proposition \ref{p:cont_criterion}, it follows that $T_+ = T(0) = \infty$. 

By \eqref{eq:sigmainfinity_bound}, \eqref{eq:ode} and \eqref{eq:sigmadot_bound1}
\begin{align}
	\sigma - t &= \sigma_0 + \int_0^t \sigma'(\bar t) d\bar t - t \\
	&\geq \sigma_0 + \sigma_1\Bigl ( 1 - \frac{\eps}{N_1 - \tau} \Bigr )(1 - C_1\eps) t - t \\
	&\geq \sigma_0 + \frac{1}{2} (\sigma_1 - 1)t
\end{align}
for all $\epsilon$ sufficiently small. A similar argument also proves there exists a constant $c > 0$  depending on $\sigma_0$ and $\sigma_1$ such that
$(1 - c)\sigma \geq t$ and thus for all $s \in [\sigma,\infty)$
\begin{align}
(s-t)^{-r-1}s^r \leq c^{-r}, 
\end{align}
for all $\epsilon$ sufficiently small. 
Thus, by \eqref{eq:finite_speed}, \eqref{eq:chis_estimate}, \eqref{eq:chit_estimate} and \eqref{eq:chiss_estimate} we conclude there exists $C > 0$ such that for all $s \in [\sigma,\infty)$
\begin{align}
	s^{r+1} [|\chi_{ss}(s,t)| + |\chi_{s,t}(s,t)|] &\leq C B, \\
	|\chi_s(s,t) - \tau| + 	|\chi_t(s,t) - \sigma_1^\infty(N_1 - \tau)| 
	&\leq C B (1 + t)^{-r}.
\end{align}
This proves the first two estimates in \eqref{eq:asymptotic} which then immediately imply the third and fourth estimates in \eqref{eq:asymptotic} via \eqref{eq:ode} and \eqref{eq:inext_definition}. 
\end{proof}

\begin{thm}\label{t:asymptotic_stability_poszeta}
	Assume $\zeta > 0$. Let $(\sigma_0, \sigma_1) \in (0,\infty) \times (1,\infty)$.  There exist $\eps_0  > 0$ and $C > 0$ such that for all $\eps < \eps_0$, the following is true. Suppose  
	$(\chi_0,\chi_1,N_0)$ is a set of initial data with initial shock front $\sigma_0$ such that there exists $r > 0$ such that 
	\begin{align}
		B:= \sup_{s \in [\sigma_0,\infty)} \Bigl [	| &\chi_0'(s) - \tau | + s^{r+2}|\chi_0''(s)| \\
		&+
		| \chi_1(s) - \sigma_1(N_1 - \tau) | + 
		s^{r+2} |\chi_1'(s)| \Bigr ]
		< \eps. \label{eq:pertubation_g02}
	\end{align}
Then the unique solution $(\chi,N)$ to the equations of motion with initial data $(\chi_0, \chi_1, N_0)$ satisfies $T_+ = T(\zeta)$ and 
the state variables and shock speed remain close to and asymptotically approach those of a piece-wise constant stretched motion as $t \rar T(\zeta)$: there exist $\sigma_1^\infty \in (1,\infty)$ such that $\bigl |\frac{\sigma_1^\infty}{\sigma_1} - 1\bigr | \leq C \eps$ and for all $t \in [0,T_+(\zeta))$ 
\begin{align}
	\begin{split}
			\sup_{s \in [\sigma(t), \infty)}& s^{r+2} [ |\chi_{ss}(s,t)| + |\chi_{st}(s,t)| ] \leq C B, \\
		\sup_{s \in [\sigma(t), \infty)}& \Bigl [ |\chi_s(s,t) - (\zeta t + \tau)| + 
		 |\chi_t(s,t) - [\zeta(s - \sigma_0)+\sigma_1^\infty (N_1 - \tau)]|
		 \Bigr ] \\
		&\qquad \leq C B (T(\zeta) - t)^{1+r}, \\
				&\left |  \sigma'(t)\Bigl (\frac{\zeta(\sigma - \sigma_0) + \sigma_1^\infty(N_1 - \tau)}{N_1 - (\zeta t + \ta)}\Bigr )^{-1} - 1
		\right | \leq C B (T(\zeta) - t)^{r}, \\
		\sup_{s \in [0,\sigma(t)]}&
		\left |
		N(s,t)\Bigl (\frac{[\zeta(\sigma - \sigma_0) + \sigma^\infty_1(N_1 - \tau)]^2}{N_1 - (\zeta t + \tau)} \Bigr )^{-1} - 1
		\right | \leq C B (T(\zeta) - t)^r.
		\end{split}\label{eq:asymptotic2} 
\end{align}
\end{thm}

\begin{proof}
As in the previous proof, it follows from \eqref{eq:pertubation_g02} and the fundamental theorem of calculus that there for all $\epsilon$ sufficiently small, there exists $\sigma_1^\infty \in (1,\infty)$ with $|\frac{\sigma_1^\infty}{\sigma_1} - 1| \leq \frac{\eps}{N_1 -\tau}$ such that for all $s \in [\sigma_0,\infty)$ 
\begin{align}
\begin{split}
|\chi_0'(s) - \tau| &\leq \frac{B}{r+1} s^{-1-r}, \\
|\chi_1(s) - (\zeta(s- \sigma_0) + \sigma_1^\infty(N_1 - \tau))| &\leq \frac{B}{r+1} s^{-1-r}.
\end{split}\label{eq:data_pert} 
\end{align} 
In particular, there exists $C_1 > 0$ depending on $\sigma_0, \sigma_1, \zeta$ and $N_1 - \tau$ such that 
\begin{align}
\Bigl |
\frac{\sigma'(0)}{\sigma_1^\infty} - 1
\Bigr | \leq C_1 B T(\zeta)^{r} \label{eq:initial_sigmaprime_pert}
\end{align}
Let $(\chi,N)$ be the unique solution to the equations of motion with shock front $\sigma$ and initial data $(\chi_0, \chi_1, N_0)$ defined on a maximal time interval of existence $[0,T_+)$.
Let $\bar \chi$ be the unique $C^2$ function on $\cl D_{\sigma_0} := \{(s,t) : t \in [0,T(\zeta)), s \in [\sigma_0 + t, \infty)\}$ solving the wave equation
\begin{align}
	\begin{split}
		&\bar \chi_{tt}(s,t) = \bar \chi_{ss}(s,t),  \\
		&\bar \chi_s(\infty,t) = \zeta t + \tau, 
	\end{split}\label{eq:wave_equation2}\\
	&\bar \chi(s,0) = \chi_0(s), \bar \chi_t(s,0) = \chi_1(s).
\end{align}
As before, using \eqref{eq:data_pert} and D'Alembert's formula, we conclude the bounds for all $t \in [0, T(\zeta)), s \in [\sigma(t),\infty)$,
\begin{align}
\begin{split}
|\bar \chi_s(s,t) - (\zeta t + \tau)| &\leq \frac{2B}{r+1}(\sigma - t)^{-1-r}, \\
|\bar \chi_t(s,t) - [\zeta(s-\sigma_0) + \sigma_1^\infty(N_1 - \tau)]| &\leq \frac{2B}{r+1}(\sigma - t)^{-1-r}, \\ 
s^{r+1}[|\bar \chi_{ss}(s,t)| + |\bar \chi_{st}(s,t)|] &\leq B ( 1 + 
(s-t)^{-r-2}s^{r+2} ), 
\end{split}\label{eq:chi_pert}
\end{align}

To control the shock speed and state variables, we proceed via a bootstrap argument. We claim that there exists a constant $C_0 \geq C_1$ depending only on $\sigma_0,\sigma_1,\zeta$ and $N_1 - \tau$ such that for all $\epsilon$ sufficiently small, the following is true: if $T < T_+$ and for all $t \in [0,T]$ we have 
\begin{align}
	\left |
	\sigma'(t)\Bigl (\frac{\zeta(\sigma(t) - \sigma_0) + \sigma_1^\infty(N_1 - \tau)}{N_1 - (\zeta t + \ta)}\Bigr )^{-1} - 1
	\right | \leq 2 C_0 B (T(\zeta) - t)^{r}, \label{eq:bootstrap}
\end{align}
then for all $t \in [0,T]$ we have 
\begin{align}
	\left |
	\sigma'(t)\Bigl (\frac{\zeta(\sigma(t) - \sigma_0) + \sigma_1^\infty(N_1 - \tau)}{N_1 - (\zeta t + \ta)}\Bigr )^{-1} - 1
	\right | \leq C_0 B (T(\zeta) - t)^{r}. \label{eq:boot_conclusion}
\end{align}
Assume that \eqref{eq:bootstrap} holds with $C_0$ to be determined. Let 
\begin{align}
	\delta(t) &:= \sigma'(t)\Bigl (\frac{\zeta(\sigma - \sigma_0) + \sigma_1^\infty(N_1 - \tau)}{N_1 - (\zeta t + \ta)}\Bigr )^{-1}- 1, \\
	f(t) &:= 
	\int_0^t \frac{\zeta (1 +  \delta(\bar t))}{N_1 - (\zeta \bar t + \tau)} d \bar t,
\end{align}
so 
\begin{align}
	f'(t) = \frac{\zeta(1 + \delta(t))}{N_1 - (\zeta t + \tau)}.
\end{align}
Then 
\begin{align}
	(\sigma - \sigma_0)' - \frac{\zeta(1 +\delta(t))}{N_1 -(\zeta t + \tau)} (\sigma- \sigma_0) = \frac{\sigma_1^\infty(N_1 -\tau)}{N_1 - (\zeta t + \tau)}(1 + \delta(t))
\end{align}
which implies 
\begin{align}
	\frac{d}{dt}[(\sigma - \sigma_0)\exp(-f(t))] = \frac{\sigma_1^\infty(N_1 - \tau)}{-\zeta} \frac{d}{dt} \exp(-f(t)).
\end{align}
Integrating, we conclude that 
\begin{align}
\sigma - \sigma_0 &= \frac{\sigma_1^\infty(N_1 - \tau)}{\zeta}(\exp f(t) - 1) \\
&= \frac{\sigma_1^\infty(N_1 - \tau)}{\zeta (N_1 - (\zeta t + \tau))}
\bigl (
(N_1 - \tau) \exp g(t) - (N_1 - (\zeta t + \tau))
\bigr ) \\
&= \frac{\sigma_1^\infty T(\zeta)}{T(\zeta) - t}
\bigl (t + 
T(\zeta) (\exp g(t) - 1)
\bigr )
\end{align}
where 
\begin{align}
	g(t) := \int_0^t \frac{\zeta \delta(\bar t)}{N_1 - (\zeta \bar t + \tau)} d \bar t
\end{align}
satisfies, via \eqref{eq:bootstrap}, 
\begin{align}
|g(t)| \leq \frac{2 C_0 (T(\zeta))^r}{r} \eps. \label{eq:g_bound}
\end{align}
Then for all $\epsilon < \frac{r}{8C_0 \sigma_1^\infty T(\zeta)^{r+1}}\min \{\sigma_0, T(\zeta)(\sigma_1^\infty - 1)\}$, 
\begin{align}
(T(\zeta) - t)(\sigma - t) &= \sigma_0(T(\zeta) - t) + \sigma_1^\infty T(\zeta)(t + 
T(\zeta)(\exp g(t) - 1)) \\&\quad - T(\zeta)t + t^2  \\
&\geq \sigma_0(T(\zeta) - t) + T(\zeta)(\sigma_1^\infty - 1)t + \sigma_1^\infty T(\zeta)^2 (\exp g(t) - 1) \\
&\geq T(\zeta) \min \{\sigma_0, T(\zeta)(\sigma_1^\infty - 1) \} - \frac{4 \sigma_1^\infty T(\zeta)^{r+2} C_0}{r} \eps \\
&\geq \frac{1}{2} T(\zeta) \min \{\sigma_0, T(\zeta)(\sigma_1^\infty - 1) \}
\end{align}
Thus, for all $p > 0$
\begin{align}
	(\sigma - t)^{-p} \leq \Bigl ( \frac{1}{2}
	T(\zeta) \min \{\sigma_0, T(\zeta)(\sigma_1^\infty - 1)
	\} 
	\Bigr )^{-p} (T(\zeta) - t)^{-p}. \label{eq:sigmat_relation}
\end{align}
Inserting \eqref{eq:sigmat_relation} into \eqref{eq:chi_pert} implies that there exists an explicit constant 
$C_2 > 0$ depending on $\sigma_0, \sigma_1, \zeta$ and $N_1 - \tau$ such that
\begin{align}
\begin{split}
\left | \frac{\bar \chi_t(\sigma, t)}{\zeta(\sigma - \sigma_0) + \sigma_1^\infty(N_1 - \tau)} - 1 \right | &\leq C_2 B (T(\zeta) - t)^{1 + r}, \\
\left | \frac{N_1 - \bar \chi_s(\sigma, t)}{N_1 - (\zeta t + \tau)} - 1 \right | &\leq C_2 B (T(\zeta) - t)^{r}. 
\end{split}\label{eq:chi_pert_Tzeta}
\end{align} 
Thus, via the relation $\sigma' = \frac{\bar \chi_t(\sigma,t)}{N_1 - \bar \chi_s(\sigma,t)}$ we conclude that there exists an explicit constant 
$C_3 > 0$ depending on $\sigma_0, \sigma_1, \zeta$ and $N_1 - \tau$ such that for all $\eps$ sufficiently small 
\begin{align}
	\left |
\sigma'(t)\Bigl (\frac{\zeta(\sigma - \sigma_0) + \sigma_1^\infty(N_1 - \tau)}{N_1 - (\zeta t + \ta)}\Bigr )^{-1} - 1
\right | \leq C_3 B (T(\zeta) - t)^{r}
\end{align}
If we define $C_0 := \max\{C_1, C_3\}$, then we have proved that \eqref{eq:bootstrap} implies \eqref{eq:boot_conclusion} for all $\epsilon < \frac{r}{8C_0 \sigma_1^\infty T(\zeta)^{r+1}}\min \{\sigma_0, T(\zeta)(\sigma_1^\infty - 1)\}$ sufficiently small. This proves the claim. 

By the claim and a continuity argument, if follows that the bound \eqref{eq:bootstrap} and its consequences hold on $[0,T_+)$ for all $\epsilon$ sufficiently small. By \eqref{eq:bootstrap}, \eqref{eq:chi_pert_Tzeta} and Proposition \ref{p:cont_criterion}, it follows that $T_+ = T(\zeta)$ for all $\epsilon$ sufficiently small. Then \eqref{eq:bootstrap}, \eqref{eq:chi_pert}, \eqref{eq:sigmat_relation}, \eqref{eq:inext_definition} and the fact $\chi |_{\cl E_\sigma} = \bar \chi|_{\cl E_{\sigma}}$ immediately yield \eqref{eq:asymptotic2} for all $\eps$ sufficiently small. 
\end{proof}

\bibliographystyle{plain}
\bibliography{researchbibmech}
\bigskip

\centerline{\scshape Casey Rodriguez}
\smallskip
{\footnotesize
 \centerline{Department of Mathematics, Massachusetts Institute of Technology}
\centerline{77 Massachusetts Ave, 2-246B, Cambridge, MA 02139, U.S.A.}
\centerline{\email{caseyrod@mit.edu}}
}

\end{document}